\documentclass[11pt,a4paper]{amsart}
\usepackage{amsmath, amsthm, amsfonts}
\usepackage[colorlinks=true,citecolor=black,linkcolor=black]{hyperref}
\usepackage[lmargin=25mm,rmargin=25mm,bmargin=25mm,tmargin=25mm]{geometry}
\usepackage[numbers,sort&compress]{natbib}
\newcommand{\doi}[1]{\href{http://dx.doi.org/#1}{doi:\texttt{#1}}}

\newcommand{\urlprefix}{}
\newcommand{\spacing}[1]{
\renewcommand{\baselinestretch}{#1}
\setlength{\footnotesep}{\baselinestretch\footnotesep}}
\spacing{1.2}
\theoremstyle{plain}
\newtheorem{theorem}{Theorem}
\newtheorem{lemma}[theorem]{Lemma}

\newtheorem{proposition}[theorem]{Proposition}

\newcommand{\B}{\ensuremath{\mathcal{B}}}

\DeclareMathOperator{\tw}{tw}

\begin{document}

\title[Nordhaus-Gaddum for Treewidth]{Nordhaus-Gaddum for Treewidth}

\author{Gwena\"el Joret} \address{\newline D\'epartement
  d'Informatique \newline Universit\'e Libre de Bruxelles \newline
  Brussels, Belgium} \email{gjoret@ulb.ac.be}

\author{David~R.~Wood} \address{\newline Department of Mathematics and
  Statistics \newline The University of Melbourne \newline Melbourne,
  Australia} \email{woodd@unimelb.edu.au}


\thanks{This work was supported in part by the Actions de Recherche
  Concert\'ees (ARC) fund of the Communaut\'e fran\c{c}aise de
  Belgique.  Gwena\"el Joret is a Postdoctoral Researcher of the Fonds
  National de la Recherche Scientifique (F.R.S.--FNRS), and is also
  supported by an Endeavour Fellowship from the Australian Government.
  David Wood is supported by a QEII Research Fellowship from the
  Australian Research Council (ARC)}

\date{\today}

\begin{abstract}
  We prove that for every $n$-vertex graph $G$, the treewidth of $G$
  plus the treewidth of the complement of $G$ is at least $n-2$. This
  bound is tight.
\end{abstract}

\maketitle

\bigskip Nordhaus-Gaddum-type theorems establish bounds on
$f(G)+f(\overline G)$ for some graph parameter $f$, where
$\overline{G}$ is the complement of a graph $G$. The literature has
numerous examples; see
\citep{BBFHH,FKSSW05,NG56,GHW92,SS06,Stiebitz-DM92,RT00} for a
few. Our main result is the following Nordhaus-Gaddum-type theorem for
treewidth\footnote{While treewidth is normally defined in terms of
  tree decompositions (see \citep{Diestel4}), it can also be defined
  as follows. A graph $G$ is a \emph{$k$-tree} if $G\cong K_{k+1}$ or
  $G-v$ is a $k$-tree for some vertex $v$ whose neighbours induce a
  $k$-clique. Then the \emph{treewidth} of a graph $G$ is the minimum
  integer $k$ such that $G$ is a spanning subgraph of a $k$-tree.  See
  \citep{Bodlaender-TCS98,Reed97} for surveys on treewidth.

  Let $G$ be a graph. Two subsets of vertices $A$ and $B$ in $G$
  \emph{touch} if $A\cap B\neq\emptyset$, or some edge of $G$ has one endpoint
  in $A$ and the other endpoint in $B$. A \emph{bramble} in $G$ is a
  set of subsets of $V(G)$ that induce connected subgraphs and pairwise touch.  A set $S$ of
  vertices in $G$ is a \emph{hitting set} of a bramble \B\ if $S$
  intersects every element of \B. The \emph{order} of \B\ is the
  minimum size of a hitting set. \citet{SeymourThomas-JCTB93} proved
  the Treewidth Duality Theorem, which says that a graph $G$ has
  treewidth at least $k$ if and only if $G$ contains a bramble of
  order at least $k+1$.}, which is a graph parameter of particular
importance in structural and algorithmic graph theory. Let $\tw(G)$
denote the treewidth of a graph $G$.

\begin{theorem}
  \label{Main}
  For every graph $G$ with $n$
  vertices, $$\tw(G)+\tw(\overline{G})\geq n-2\enspace.$$
\end{theorem}

The following lemma is the key to the proof of Theorem~\ref{Main}.

\begin{lemma}
  \label{General}
  Let $G$ be a graph with $n$ vertices, no induced 4-cycle, and no
  $k$-clique.  Then $\tw(\overline G)\geq n-k$.
\end{lemma}

 \begin{proof}
   Let $\B:=\{\{v,w\}:vw\in E(\overline{G})\}$. If
   $\{v,w\}$ and $\{x,y\}$ do not touch for some $vw,xy\in
   E(\overline{G})$, then the four endpoints are distinct and
   $(v,x,w,y)$ is an induced 4-cycle in $G$, which is a contradiction.
   Thus $\B$ is a bramble in $\overline{G}$. Let $S$ be a hitting set
   for $\B$. Thus no edge in $\overline{G}$ has both endpoints in
   $V(\overline{G})\setminus S$. Hence $V(G)\setminus S$ is a clique
   in $G$. Therefore $n-|S|\leq k-1$ and $|S|\geq n-k+1$.  That is,
   the order of $\B$ is at least $n-k+1$.  By the Treewidth Duality
   Theorem, $\tw(\overline{G})\geq n-k$, as desired.
 \end{proof}

\begin{proof}[Proof of Theorem~\ref{Main}] 
  Let $k:=\tw(G)$. Let $H$ be a $k$-tree that contains $G$ has a
  spanning subgraph. Thus $H$ has no induced 4-cycle (it is chordal)
  and has no $(k+2)$-clique. By Lemma~\ref{General} and since
  $\overline G \supseteq \overline H$, we have $\tw(\overline G)\geq
  \tw(\overline H)\geq n-k-2$.  That is, $\tw(G)+\tw(\overline G)\geq
  n-2$.
\end{proof}

Lemma~\ref{General} immediately implies the following result of
independent interest.

\begin{theorem}
  \label{Girth}
  For every graph $G$ with girth at least 5, we have $\tw(\overline
  G)\geq n-3$.
\end{theorem}

For $k$-trees we have the following precise result, which proves that
the bound in Theorem~\ref{Main} is tight. Let $Q_n^k$ be the $k$-tree
consisting of a $k$-clique $C$ with $n-k$ vertices adjacent only to
$C$.

\begin{theorem}
  \label{kTree}
  For every $k$-tree $G$,
$$\tw(G)+\tw(\overline{G})=
\begin{cases}
  n-1&\text{ if }G\cong Q_n^k\\
  n-2&\text{ otherwise}\enspace.
\end{cases}
$$
\end{theorem}

\begin{proof}
  First suppose that $G\cong Q_n^k$. Then $\overline{G}$ consists of
  $K_{n-k}$ and $k$ isolated vertices. Thus $\tw(\overline{G})=n-k-1$,
  and $\tw(G)+\tw(\overline{G})=n-1$.

  Now assume that $G\not\cong Q_n^k$.  By the definition of $k$-tree,
  $V(G)$ can be labelled $v_1,\dots,v_n$ such that
  $\{v_1,\dots,v_{k+1}\}$ is a clique, and for $j\in\{k+2,\dots,n\}$,
  the neighbourhood of $v_j$ in $G[\{v_1,\dots,v_{j-1}\}]$ is a
  $k$-clique $C_j$.  If $C_{k+2},\dots,C_n$ are all equal then $G\cong
  Q_n^k$. Thus $C_j\neq C_{k+2}$ for some minimum integer $j$. Observe
  that each vertex in $C_j$ has a neighbour outside of $C_j$.
  Arbitrarily label $C_j=\{x_1,\dots,x_{k+1}\}$, and let $y_i$ be a
  neighbour of each $x_i$ outside of $C_j$.

  We now describe an $(n-k-2)$-tree $H$ that contains $\overline
  G$. Let $A:=V(G)\setminus C_j$ be the starting $(n-k-1)$-clique of
  $H$. Add each vertex $x_i$ to $H$ adjacent to
  $A\setminus\{y_i\}$. Observe that $H$ is an $(n-k-2)$-tree and $\overline G$ is a spanning
  subgraph of $H$.
  Thus $\tw(\overline G)\leq n-k-2$ and $\tw(G)+\tw(\overline{G})\leq
  n-2$, with equality by Theorem~\ref{Main}.
\end{proof}

In view of Theorem~\ref{Main}, it is natural to also consider how
large $\tw(G) + \tw(\overline G)$ can be.  Every $n$-vertex graph $G$
satisfies $\tw(G)\leq n-1$, implying $\tw(G) + \tw(\overline G) \leq
2n - 2$. It turns out that this trivial upper bound is tight up to
lower order terms.  Indeed, \citet{PS11} proved that, if
$G\in\mathcal{G}(n, p)$ is a random $n$-vertex graph with edge
probability $p =\omega(\frac{1}{n})$ in the sense of Erd\H{o}s and
R\'enyi, then asymptotically almost surely $\tw(G)=n-o(n)$; see
\citep{KB92,LLO10} for related results. Setting $p=\frac12$, it
follows that asymptotically almost surely, $\tw(G)=n-o(n)$ and
$\tw(\overline G) = n - o(n)$, and hence $\tw(G) + \tw(\overline G) =
2n - o(n)$.

\medskip Theorems~\ref{Main} and \ref{kTree} can be reinterpreted as
follows.

\begin{proposition}
  \label{CliqueSize}
  For all graphs $G_1$ and $G_2$, the union
  $G_1\cup G_2$ contains no clique on $\tw(G_1)+\tw(G_2)+3$ vertices.
  Conversely, there exist graphs $G_1$ and $G_2$   such that $G_1\cup G_2$ contains a clique on $\tw(G_1)+\tw(G_2)+2$
  vertices.
\end{proposition}

\begin{proof}
For the first claim, we may assume that $V(G_1)=V(G_2)$ and $E(G_1)\cap E(G_2)=\emptyset$.  Let $S$ be a
  clique in $G_1\cup G_2$.  Thus $G_1[S]$ and $G_2[S]$ are
  complementary.  By Theorem~\ref{Main}, $\tw(G_1)+\tw(G_2)\geq
  \tw(G_1[S])+\tw(G_2[S])\geq |S|-2$. Thus $|S|\leq
  \tw(G_1)+\tw(G_2)+2$ as desired.  The converse claim follows from
  Theorem~\ref{kTree}.
\end{proof}

Proposition~\ref{CliqueSize} suggests studying $G_1\cup G_2$
further. For example, what is the maximum of $\chi(G_1\cup G_2)$ taken
over all graphs $G_1$ and $G_2$ with $\tw(G_1)\leq k$ and
$\tw(G_2)\leq k$? By Proposition~\ref{CliqueSize} the answer is at
least $2k+2$. A minimum-degree greedy algorithm proves that
$\chi(G_1\cup G_2)\leq 4k$. This question is somewhat similar to
Ringel's earth--moon problem which asks for the maximum chromatic
number of the union of two planar graphs.


\def\soft#1{\leavevmode\setbox0=\hbox{h}\dimen7=\ht0\advance \dimen7
  by-1ex\relax\if t#1\relax\rlap{\raise.6\dimen7
    \hbox{\kern.3ex\char'47}}#1\relax\else\if T#1\relax
  \rlap{\raise.5\dimen7\hbox{\kern1.3ex\char'47}}#1\relax \else\if
  d#1\relax\rlap{\raise.5\dimen7\hbox{\kern.9ex
      \char'47}}#1\relax\else\if D#1\relax\rlap{\raise.5\dimen7
    \hbox{\kern1.4ex\char'47}}#1\relax\else\if l#1\relax
  \rlap{\raise.5\dimen7\hbox{\kern.4ex\char'47}}#1\relax \else\if
  L#1\relax\rlap{\raise.5\dimen7\hbox{\kern.7ex
      \char'47}}#1\relax\else\message{accent \string\soft \space #1
    not defined!}#1\relax\fi\fi\fi\fi\fi\fi}

\end{document}